\documentclass[11pt]{amsproc}
\usepackage{mathrsfs,amssymb,amsmath,amsrefs,amsfonts,amsthm,anysize,enumerate,float,times,amsmath,hyperref,url}
\hypersetup{citecolor=red, linkcolor=blue, colorlinks=true}

\usepackage{amssymb,latexsym,verbatim,pstricks,tikz,stmaryrd} 
\usepackage{enumerate,datetime,booktabs}
\usepackage{hyperref,url,pst-node}
\usetikzlibrary{calc}
\hypersetup{citecolor=red, linkcolor=blue, colorlinks=true}

\theoremstyle{plain}
\newtheorem{theorem}{Theorem}[section]

\newtheorem{proposition}[theorem]{Proposition}
\newtheorem{lemma}[theorem]{Lemma}

\newtheorem{hypothesis}[theorem]{Hypothesis}

\newtheorem{problem}[theorem]{Problem}

\theoremstyle{remark}

\newcommand{\s}{\mathcal{S}}
\newcommand{\p}{\mathcal{P}}

\renewcommand{\le}{\leqslant}
\renewcommand{\ge}{\geqslant}
\renewcommand{\leq}{\leqslant}
\renewcommand{\geq}{\geqslant}
\newcommand{\Aut}{\textup{Aut}}
\newcommand{\GL}{\mathrm{GL}}
\newcommand{\PSL}{\mathrm{PSL}}
\newcommand{\PSU}{\mathrm{PSU}}
\newcommand{\PSP}{\mathrm{PSp}}
\newcommand{\Sp}{\mathrm{Sp}}
\newcommand{\SO}{\mathrm{SO}}
\newcommand{\SU}{\mathrm{SU}}
\newcommand{\SL}{\mathrm{SL}}
\newcommand{\soc}{\mathrm{soc}}
\newcommand{\GamL}{\Gamma\mathrm{L}}
\newcommand{\PGamL}{\mathrm{P}\Gamma\mathrm{L}}
\newcommand{\lhdeq}{\trianglelefteqslant}    

\title[Hexagons, octagons and projective linear groups]{Point-primitive generalised hexagons and octagons\\
and projective linear groups}

\author{S. P. Glasby, E. Pierro, Cheryl E. Praeger}
\date{\currenttime\ \today}       
\address{
\phantom{|}\kern-1cm
Centre for the Mathematics of Symmetry and Computation\\
Department of Mathematics and Statistics\\
The University of Western Australia\\
35 Stirling Highway, Perth, WA 6009, Australia.\newline
{{\sc Email addresses:} \tt \{stephen.glasby,emilio.pierro,cheryl.praeger\}@uwa.edu.au}
}

\subjclass[2010]{primary 51E12; secondary 20B15, 05B25}

\keywords{generalised hexagon, generalised octagon, generalised polygon, primitive permutation group}

\begin{document}

\begin{abstract}
We discuss recent progress on the problem of classifying point-primitive generalised polygons. In the case of generalised hexagons and generalised octagons, this has reduced the problem to primitive actions of 
almost simple groups of Lie type.
To illustrate how the natural geometry of these groups may be used in this study, we show that if $\s$ is a finite thick generalised hexagon or octagon with $G \le\Aut(\s)$ acting point-primitively and the socle of $G$ isomorphic to $\PSL_n(q)$ where $n \ge 2$, then the stabiliser of a point acts irreducibly on the natural module.
We describe a strategy to prove that such a generalised hexagon or octagon $\s$ does not exist.
\end{abstract}

\maketitle

\section{Introduction}
We show in this paper that the Aschbacher--Dynkin \cite{asch} classification of maximal subgroups is a potentially useful tool to investigate whether or not a finite thick generalised hexagon or octagon admits a large rank classical group with a point-primitive action.

The notion of a generalised polygon arose from the investigations of Tits \cite{tits59} and is connected with the groups of Lie type having twisted Lie rank $2$.
They belong to a wider class of geometric objects known as buildings, which were also introduced by Tits, whose motivation was to find natural geometric objects on which the finite groups of Lie type act, in order to work towards a proof of the classification of finite simple groups.
Indeed, all families of simple groups of Lie type having twisted Lie rank $2$ arise as automorphism groups of generalised polygons.
The groups $A_2(q) \cong \PSL_3(q)$ are admitted by generalised triangles (projective planes); the groups $^2A_3(q) \cong \PSU_4(q)$, $^2A_4(q) \cong \PSU_5(q)$, $B_2(q) \cong O_5(q)$, $C_2(q) \cong \PSP_4(q)$ and $^2D_3(q) \cong O_6^-(q)$ are admitted by generalised quadrangles; the groups $G_2(q)'$ and $^3D_4(q)'$ are admitted by generalised hexagons; and the groups $^2F_4(2^{2m+1})'$ are admitted by generalised octagons.
Isomorphisms between these groups, namely $B_2(q) \cong C_2(q)$ and $^2A_3(q) \cong {}^2D_3(q)$, lead to  these particular groups acting on more than one geometry.

We refer the reader to Van Maldeghem's book~\cite{van} both for further details about these classical generalised polygons, and for a full introduction to the theory of generalised polygons.
An \emph{incidence geometry $\s = (\p,\mathcal{L},\mathcal{I})$ of rank 2} consists of a point set $\p$, a line set $\mathcal{L}$ and an incidence relation $\mathcal{I} \subseteq \p \times \mathcal{L}$ such that $\p$ and $\mathcal{L}$ are disjoint non-empty sets.
We say that $\s$ is finite if $|\p \cup \mathcal{L}|$ is finite.
The \emph{dual} of $\s$ is $\s^D = (\mathcal{L},\p,\mathcal{I}^D)$, where $(p,\ell) \in \mathcal{I}$ if and only if $(\ell,p) \in \mathcal{I}^D$.
We say that $\s$ is \emph{thick} if each point is incident with at least three lines and each line is incident with at least three points.
A \emph{flag} of $\s$ is a set $\{p,\ell\}$ with $p \in \p$, $\ell \in \mathcal{L}$ and $(p,\ell) \in \mathcal{I}$.
The \emph{incidence graph} of $\s$ is the bipartite graph whose vertices are $\p \cup \mathcal{L}$ and whose edges are the flags of $\s$.
A \emph{generalised $n$-gon} is, then, a thick incidence geometry of rank 2 whose incidence graph is connected of diameter $n$ and girth $2n$ such that each vertex lies on at least three edges \cite[Lemma 1.3.6]{van}.
It is not immediate, but if $\s$ is a generalised $n$-gon, then there are constants $s, t \ge 2$ such that each point is incident with $t+1$ lines and each line is incident with $s+1$ points \cite[Corollary 1.5.3]{van}.
We then say that the \emph{order} of $\s$ is $(s,t)$.
A \emph{collineation} of $\s$ is a pair $(\alpha,\beta) \in\textup{Sym}(\p)\times\textup{Sym}(\mathcal{L})$ that preserves the subset $\mathcal{I} \subseteq \p \times \mathcal{L}$.
The subset of all collineations of $\textup{Sym}(\p)\times\textup{Sym}(\mathcal{L})$ is a subgroup denoted $\Aut(\s)$.
A celebrated result of Feit and Higman \cite{feithigman} states that if $\mathcal{S}$ is a finite thick generalised $n$-gon, then $n \in \{2,3,4,6,8\}$, and the square-free part of $st$ is $1$ if $n\in\{3,6\}$ and $2$ if $n=8$.

The classification of (not necessarily thick) generalised polygons admitting an automorphism group which acts distance-transitively on the \emph{point graph} of $\s$ is due to Buekenhout and Van Maldeghem \cite{bvm}. This is the graph with points as vertices and with two points adjacent if they are collinear. 
The assumption of distance-transitivity for this graph is strong.
If $\s$ is also thick, then Buekenhout and Van Maldeghem show that the socle of $G$ is one of the aforementioned Lie type groups and $\s$ is classical.
In addition, they show that distance-transitivity implies that $G$ acts primitively on $\p$.
In recent years there has been work by a number of authors to show that the assumption of distance-transitivity can be relaxed.

In this paper we focus on the cases $n=6$ or $8$, that is, the generalised hexagons and octagons, respectively, but first we
 mention briefly what is known about $\Aut(\s)$ for the other generalised polygons.
In the case $n=2$, the incidence graph of $\mathcal{S}$ is a complete bipartite graph and so $\Aut(\s) \cong\textup{Sym}(\p) \times\textup{Sym}(\mathcal{L})$.
In the case $n=3$, $\mathcal{S}$ is a projective plane and the recent characterisation due to Gill \cite{gill07} is currently the best result.
It is shown by Gill that if $G$ acts transitively on the points of a generalised $3$-gon $\s$, then either $\s$ is Desarguesian or each minimal normal subgroup of $G$ is elementary abelian.
In the case $n=4$, Bamberg, Giudici, Morris, Royle and Spiga \cite{Bamberg2012gq} prove that if $\mathcal{S}$ is finite and thick, and $G$ acts primitively on both $\p$ and $\mathcal{L}$, then $G$ is almost simple.
If in addition $G$ acts flag-transitively, then $G$ is almost simple of Lie type.
Weakening the hypothesis significantly, by removing flag-transitivity and by relaxing the assumption of line-primitivity to line-transitivity, Bamberg, Popiel and the third author were able to show in \cite{bpp17} that such a group $G$ has a unique minimal normal subgroup.
They obtain many further restrictions in \cites{JCD16,Nag19} and conjecture that a point-primitive group acting on a finite thick generalised quadrangle must be of affine or almost simple type \cite[Conjecture 1.7]{Nag19}.


Henceforth, $\s$ will denote a finite thick generalised hexagon or octagon.
At present the only known examples of thick generalised hexagons and octagons are the split Cayley hexagon $H(q)$, the twisted triality hexagon $T(q,q^3)$, the Ree--Tits octagon $O(2^{2m+1})$ and the duals of each of them.
These correspond to the groups $G_2(q)$, $^3D_4(q)$ and $^2F_4(2^{2m+1})$ and complete descriptions of these can be found in \cite{van}.
Schneider and Van Maldeghem \cite[Theorem 2.1]{schneivanmal} showed that if $G \le$ $\Aut(\s)$ acts flag-transitively, point-primitively and line-primitively, then $G$ is an almost simple group of Lie type.
The following theorem, which significantly strengthened this result, provided motivation for the present paper.

\begin{theorem}[\cite{bgpps}*{Theorem 1.2}] \label{motiv}
Let $\mathcal{S}$ be a finite thick generalised hexagon or octagon.
If a subgroup $G$ of $\Aut(\mathcal{S})$ acts point-primitively, then $G$ is an almost simple group of Lie type. 
\end{theorem}

The proof of Theorem \ref{motiv} relies on the classification of finite simple groups.
In order to rule out certain possibilities for $\soc(G)$, it is sufficient to consider the primitive actions of the almost simple groups of Lie type, or equivalently, their maximal subgroups.
For an exceptional Lie type group that has a faithful projective representation in defining characteristic of degree at most 12,  a 
complete classification of its maximal subgroups is  summarised in~\cite[Chapter 7]{bhrd}.
Using this classification it was proved by Morgan and Popiel in \cite{mopo} that under the hypothesis of the above theorem, if in addition it is assumed that the socle of $G$ is isomorphic to one of the Suzuki--Ree groups, $^2B_2(2^{2m+1})'$, $^2G_2(3^{2m+1})'$ or $^2F_4(2^{2m+1})'$, where $m \ge 0$, then up to point-line duality, $\mathcal{S}$ is the Ree--Tits octagon $O(2^{2m+1})$.
For a general classical group $G$, however, we appeal to the Aschbacher--Dynkin classification \cite{asch} of its maximal subgroups.
The maximal subgroups of $G$ fall into eight families of ``geometric'' subgroups, those which preserve a natural geometric structure, and a ninth class of exceptions.
These classes are denoted $\mathscr{C}_i$ for $1 \le i \le 9$ and the class $\mathscr{C}_1$ consists of stabilisers of subspaces and includes the maximal parabolic subgroups of~$G$.
Our aim in this paper is to expose the utility of the Aschbacher--Dynkin classification as a tool for characterising generalised polygons whose automorphism groups have classical type socles.
Our main result is as follows.

\begin{theorem} \label{main}
Let $\mathcal{S}$ be a finite thick generalised hexagon or octagon.
If $G \le\Aut(\mathcal{S})$ acts point-primitively on $\mathcal{S}$ and the socle of $G$ is isomorphic to $\PSL_n(q)$ where $n \ge 2$, then the stabiliser of a point of $\s$ is not the stabiliser in $G$ of a subspace of the natural module $V = (\mathbb{F}_q)^n$.
\end{theorem}

The subspace stabilisers considered in Theorem~\ref{main} are all maximal parabolic subgroups. Given this result, and in the light of the result of Morgan and Popiel \cite{mopo} mentioned above, it would in the first instance be good to handle all primitive coset actions of Lie type groups on maximal parabolic subgroups. 

\begin{problem}\label{prob1} 
Extend Theorem \ref{main} to show that, if $\mathcal{S}$ is a finite thick generalised hexagon or octagon and $G \le\Aut(\mathcal{S})$ is an almost simple group of Lie type such that the stabiliser $G_x$ of a point $x$ is a maximal parabolic subgroup, then $(\s, G)$ is one of the known classical examples.
\end{problem}
Problem~\ref{prob1} has been solved for the Suzuki--Ree groups in~\cite{mopo}, and it has also been solved by Popiel and the second author~\cite{pp} for the groups $G_2(q)'$. It would be especially interesting to have a solution to Problem~\ref{prob1} for the groups of (twisted or untwisted) Lie rank 2, and in particular  for the family $^3D_4(q)'$ which is the only untreated case where the groups are known to act on a generalised hexagon or octagon.  Moreover, it would be even more interesting to have a characterisation of all point-primitive actions of groups with socle $G_2(q)'$ or $^3D_4(q)'$ on a thick generalised hexagon or octagon (not just the coset actions on maximal parabolic subgroups). This, however, seems to be a substantially harder problem.


\smallskip

Maximal parabolic subgroups mentioned in Problem~\ref{prob1} are examples of large subgroups, in the sense that  a subgroup $H$ of a finite group $G$ is \emph{large} if $|H|^3 > |G|$.
Alavi and Burness \cite{large} have determined all large subgroups of all finite simple groups.
In our view the next level of attack on the general classification problem would be to handle actions on cosets of large subgroups.

\begin{problem}\label{prob2}
Extend Theorem \ref{main} to show that, if $\mathcal{S}$ is a finite thick generalised hexagon or octagon and $G \le\Aut(\mathcal{S})$ is an almost simple group of Lie type such that the stabiliser $G_x$ of a point is a large maximal subgroup, then $(\s, G)$ is one of the known classical examples.
\end{problem}

Popiel and the second author~\cite{pp} have almost solved Problem~\ref{prob2} for  groups $G$ with socle $G_2(q)'$. The only unresolved point-primitive action is on a generalised hexagon with stabiliser satisfying $G_x \cap\soc(G) \cong G_2(q^{1/2})$. 

Building on our result Theorem~\ref{main} for parabolic actions, and applying the result of Alavi and Burness \cite[Theorem 4]{large},  we show that a solution to Problem~\ref{prob2} for groups $G\cong \PSL_n(q)$ is reduced to consideration of four kinds of point actions. 
We follow Alavi and Burness in using \emph{type} to denote a rough approximation of the structure of a subgroup.

\begin{proposition} \label{large}
Let $\mathcal{S}$ be a finite thick generalised hexagon or octagon of order $(s,t)$.
Suppose that $G \leq\Aut(\mathcal{S})$ with $G \cong \PSL_n(q)$, and $G$ acts point-primitively on $\mathcal{S}$ such that the stabiliser $G_x$ of a point $x$ is a large subgroup.  Then one of the following holds:
\begin{enumerate}[{\rm (a)}]
\item $G_x$ is a $\mathscr{C}_2$-subgroup of type $\GL_{n/k}(q) \wr S_k$, where (i) $k= 2$, or (ii) $k= 3$,  $q \in \{5, 8, 9\}$, and $(n, q-1) = 1$;
\item $G_x$ is a $\mathscr{C}_3$-subgroup of type $\GL_{n/k}(q^k)$, where (i) $k = 2$, or (ii) $k = 3$ and either $q \in \{2, 3\}$, or $q = 5$ and $n$ is odd;
\item $G_x$ is a $\mathscr{C}_5$-subgroup of type $\GL_n(q_0)$ with $q = q_0^k$ , and either $k = 2$ , or $k = 3$, or;
\item $G_x \in \mathscr{C}_8$ of type $\Sp_n(q)$ ($n$ even), $\SU_n(q_0)$ ($q=q_0^2)$, $\SO_n(q)$ ($nq$ odd), or $\SO_n^\epsilon(q)$ ($n$ even, $\epsilon=\pm$).
\end{enumerate}
\end{proposition}

\begin{proof}[Proof of Proposition \ref{large}]
In addition to the classes asserted in the statement of the proposition, Alavi and Burness show that either $G_x \in \mathscr{C}_1$, which is excluded by Theorem~\ref{main}, or $G_x$ is as in part~(a) with $(n, t, q) = (3, 3, 11)$, or $G_x$ is one of finitely many examples belonging to classes $\mathscr{C}_6$ or $\mathscr{C}_9$ \cite[Proposition 4.7 and Theorem 4(ii)]{large}.
The cases where $G$ is a group appearing in the Atlas \cite{atlas} are excluded by \cite{bvmremarks}*{Theorems 1.1 and 1.2}.
The remaining possibilities for $(G,G_x)$ are:
\begin{table}[!ht]
  \begin{tabular}{lllll}
    \toprule
    $G$&$\PSL_5(3)$&$\PSL_4(5)$&$\PSL_4(7)$&$\PSL_2(q)\ q\in\{41,49,59,61,71\}$\\ \midrule    
    $G_x$&$M_{11}$&$2^4\ldotp A_6$&$\PSU_4(2)$&$A_5$\\
  \bottomrule
  \end{tabular}
\end{table}

\noindent
The number $|\p|$ of points is the
polynomial $f(s,t)=(s+1)(s^2t^2+st+1)$ if $\mathcal{S}$ is a generalised hexagon, 
and $f(s,t)=(s+1)(s^3t^3+s^2t^2+st+1)$ if $\mathcal{S}$ is a generalised octagon.
Running through the possibilities for $|\p|=|G:G_x|$ from the table above, we find that there
are no solutions to the equation $|\p|=f(s,t)$ with $s,t\geq 2$. This completes the proof.
\end{proof}

Extending Theorem \ref{main} to include the large subgroups in class $\mathscr{C}_2$ has also proven to be unexpectedly challenging to the authors.

\section{The proof of Theorem \ref{main}}
For the proof of Theorem \ref{main} we assume that $\soc(G)=\PSL_n(q)$ and that a point stabiliser is maximal in $G$ and is 
the stabiliser of a $k$-dimensional subspace $U$ of the natural module
$V = (\mathbb{F}_q)^n$,  where $0<k<n$.  Hence we may identify the point set $\p$ of $\s$ 
with the set of $k$-subspaces of $V$, which we denote by $\binom{V}{k}$. 
In particular, if $G$ contains a graph automorphism then $k=n/2$ and, for its index $2$ subgroup 
$H=G\cap \PGamL_n(q)$, the stabiliser $H_U$ is maximal in $H$. 
Thus we may assume that $\PSL_n(q)\trianglelefteqslant G \le \PGamL_n(q)$, so that $G$ acts on
$V$.   A graph automorphism of $G$ 
maps $\binom{V}{k}$ to $\binom{V}{n-k}$, and hence maps $\s$ to an isomorphic generalised polygon with point set identified with $\binom{V}{n-k}$. This means that we need only consider $1\leq k\leq n/2$.
Moreover, to facilitate the proof of Theorem \ref{main}, we consider the action of $\SL_n(q) \trianglelefteqslant G \le \GamL_n(q)$ on $V$, with the scalar matrices acting trivially on $V$.

Working towards a contradiction, we assume the following hypothesis.

\begin{hypothesis} \label{hyp}
Let $\mathcal{S=(P,L)}$ be a finite thick generalised hexagon or octagon of order $(s,t)$,
 such that $\p$ is identified with the set $\binom{V}{k}$ of $k$-subspaces of  $V = (\mathbb{F}_q)^n$, where $1\leq k\leq n/2$.  Suppose that $\SL_n(q)\trianglelefteqslant G\leq \GamL_n(q)$  and that $G$ induces a group of automorphisms of $\s$ acting naturally on $\p$, (so that a point stabiliser belongs to class $\mathscr{C}_1$).
\end{hypothesis}

Our proof of Theorem \ref{main} uses the following three lemmas. The first is from \cite{bgpps}.

\begin{lemma}[\cite{bgpps}*{Lemma~2.1(iv)}]\label{diamondOriginal}
Let $\mathcal{S}$ be a finite thick generalised hexagon or octagon of order $(s,t)$, and let $\mathcal{P}$ denote the set of points of $\mathcal{S}$.
Let $x,y_1,y_2\in\mathcal{P}$ such that $x\sim y_1$ and $x\sim y_2$, and let $g\in\Aut(\mathcal{S})$ such that $xg\neq x$.
If $g$ fixes $y_1$ and $y_2$, then $x,y_1,y_2,xg$ all lie on a common line.
\end{lemma}

The second lemma is not difficult to prove, and its proof is left to the reader.

\begin{lemma} \label{gam}
Suppose $\SL_n(q) \lhdeq G \le \GamL_n(q)$, $V=(\mathbb{F}_q)^n$ and $k \le n/2$.
Then, if $\dim(V) = n$ and $k \le n/2$, then the orbits of $G$ on $\binom{V}{k} \times \binom{V}{k}$ are
\begin{equation}\label{E:gamma}
\Gamma_i = \left\{ (x,y) \in \binom{V}{k} \times \binom{V}{k} \mid
\dim(x \cap y) = i\right\} \quad \text{ where $0 \le i \le k$}.
\end{equation}
Moreover, for $x \in \binom{V}{k}$ the orbits of $G_x$,  are
\[
\Gamma_i(x) = \left\{y \in \binom{V}{k} \mid \dim(x \cap y)=i\right\}
\quad \text{ where $0 \le i \le k$}.
\]
\end{lemma}

The third lemma allows us to characterise adjacency in $\s$.

\begin{lemma} Assume Hypothesis~\ref{hyp}  and let $x, y \in \mathcal{P}$.
Then the following properties hold.
\begin{itemize}
\item[(F1)] For every $i \in \{ 0,\ldots,k\}$, if $x, y$ are collinear and $\dim(x \cap y) = i$, then any $x',y' \in \mathcal{P}$ with $\dim(x' \cap y') = i$ are also collinear. 
\item[(F2)] For every $i \in \{ 0,\ldots,k-1\}$, if $x, y$ are collinear and $\dim(x \cap y) = i$, then there exists $y' \in \mathcal{P}$ such that $\dim(x \cap y')=i$ and $y' \not \sim y$.
\end{itemize}

\end{lemma}
\begin{proof}
Property (F1) follows from Lemma \ref{gam}.
For (F2), suppose towards a contradiction that every point $y'$ with $\dim(x \cap y')=i$ is collinear with $y$. 
By (F1), every such point $y'$ is also collinear with $x$, and hence lies on the line $\ell$ through $x$ and $y$ (because otherwise $\{x,y,y'\}$ would form a triangle and $\mathcal{S}$ contains no triangles). 
Let $J = J(n,k)_i$ denote the generalised Johnson graph with vertex set $V(J) = \binom{V}{k}$ and two vertices adjacent if and only if they intersect in an $i$-subspace.
Since $G$ acts primitively on $\binom{V}{k}$, and since the connected components are $G$-invariant, $J$ is a connected graph.
Note that Property (F1) implies that adjacency in $J$ implies collinearity, but the converse is not necessarily true.
By definition of $J$, $y, y' \in J_1(x)$, the set of vertices adjacent to $x$ in $J$.
By the above argument, $\{x\} \cup J_1(x)$ is contained in the line $\ell$.
Since $G$ acts transitively on $J$ and since adjacency is preserved by this action, it is true for all $u \in \p$ that $\{u\} \cup J_1(u)$ is contained in a line of $\s$. Since $\s$ has more than one line, the diameter of $J$ is at least 2. 

We now prove by induction on the distance $d$, where $2 \le d \le\textup{diam}(J)$, that, for any vertices $u, v$ of $J$, if the distance $d = \delta(u,v)$ and $(u_0, u_1, \dots, u_d)$ is a path of length~$d$ in $J$ from $u=u_0$ to $v=u_d$, then $\{u_0,\dots,u_d\}$ is contained in the line $\ell$ containing $\{u\} \cup J_1(u)$.
First we prove this for $d=2$. Suppose that $\delta(u,v)=2$ and let $(u, w, v)$ be a path of length 2 in $J$ from $u$ to $v$. Note that $w\in J_1(u)\subseteq \ell$. Also $u, v$ both lie in $\{w\}\cup J_1(w)$ which, as we have shown, is contained in some line $\ell'$ of $\s$.  Then $u, w$ are contained in both $\ell$ and $\ell'$, and since two points lie in at most one line of $\s$ it follows that $\ell'=\ell$,
and so $u, w, v$ all lie in $\ell$ and the inductive assertion is proved for $d=2$. 
Now suppose inductively that $3\leq d\leq\textup{diam}(J)$ and that the assertion is true for all integers from $2$ to $d-1$.  Suppose that $\delta(u,v)=d$ and that $(u_0, u_1, \dots, u_d)$ is a path in $J$ from $u=u_0$ to $v=u_d$. Then $\delta(u,u_{d-1})=d-1$, so by induction $\{u_0,\dots,u_{d-1}\}\subseteq \ell$. Also $u_{d-2}, y\in \{u_{d-1}\}\cup J(u_{d-1})$,  which we have shown to be contained in some line $\ell'$; since $u_{d-2}, u_{d-1}$ are contained in both $\ell$ and $\ell'$, it follows that $\ell'=\ell$,
and the inductive assertion is proved for $d$. Hence by induction the assertion holds for all 
$d\leq\textup{diam}(J)$.
However, this is a contradiction because the points of $\mathcal{S}$ do not all lie on a single line.
\end{proof}

We are now in a position to prove Theorem \ref{main}.

\begin{proof}[Proof of Theorem \ref{main}]
Assume that $\s$, $G$ are as in the statement of Theorem \ref{main}   and that a point stabiliser $G_x$ is the stabiliser of a $k$-subspace. As discussed at the beginning of this section we may assume that Hypothesis~\ref{hyp} holds. Thus $\p=\binom{V}{k}$ and $k \le n/2$.

{\sc Claim 1}: $k \ge 4$.
Consider the action of $G$ on $\mathcal{P} \times \mathcal{P}$.
For each $i$ with $0 \le i \le k-1$, $G$ acts transitively on the set
$\Gamma_i$ defined in~\eqref{E:gamma} by Lemma~\ref{gam}.
It is a standard result in the theory of permutation groups that the orbits of $G$ on $\mathcal{P} \times \mathcal{P}$ are in one-to-one correspondence with the orbits of $G_x$ on $\mathcal{P}$,
and there must be at least one $G_x$-orbit for each possible distance from $x$ in the point graph of $\s$.
If $k < 3$, then the number of orbits of $G_x$ is less than four, so no point of $\mathcal{P} \setminus \{x\}$ is at distance $3$ from $x$ in $\mathcal{S}$, contradicting the assumption that $\mathcal{S}$ is either a generalised hexagon or a generalised octagon.
If $k=3$, then for the same reason $\s$ is not a generalised octagon, and so  $\mathcal{S}$ 
is a generalised hexagon and $G$ acts distance transitively on the point graph.
By the main result of Buekenhout and Van Maldeghem in \cite{bvm},  $\mathcal{S}$ is a classical generalised hexagon and its distance transitive group has socle $G_2(r^f)$ for some prime power $r^f$, which is a contradiction.
Hence $k \ge 4$ as claimed.

\smallskip

Now let $\{ e_1, \dots, e_n\}$ be a basis of $V$ and take $x = \langle e_1, \dots, e_k \rangle$. Let $k_1 < k$ be maximal such that there exists a point $y \sim x$ with $(x,y) \in \Gamma_{k_1}$ (as defined in \eqref{E:gamma}). Note that, by Claim 1, $n\geq 2k\geq 8$.

\smallskip

{\sc Claim 2}: $k_1 < k-1$.
For a contradiction, assume that $k_1 = k-1$ and without loss of generality that $y = \langle e_1, \dots, e_{k-1}, e_{k+1}\rangle$.
By (F2) there exists a point $y' \in \mathcal{P}$ such that $(x,y') \in \Gamma_{k-1}$ and $y \nsim y'$ and by (F1) we have $x \sim y'$ and so $\dim(x \cap y') = k-1$ and $\dim(y \cap y') \le k-2$.
Now $\dim(x \cap y) =\dim(x \cap y') = k-1$ implies that $\dim(x \cap y \cap y') \ge k-2$, and hence $\dim(y \cap y') =\dim(x \cap y \cap y') = k-2$.
We may assume without loss of generality that $y' = \langle e_2, \dots, e_k, e_{k+2} \rangle$.
But now the permutation matrix corresponding to $(1,k+1)(k,k+2)$ leaves $y$ and $y'$ fixed, but not $x$.
By Lemma~\ref{diamondOriginal}, this implies that $y\sim y'$, a contradiction. Hence $k_1 < k-1$, as required.

\smallskip

{\sc Claim 3}: $k_1 = 0$.
Recall that $k_1 < k$ is maximal such that there exists a point $y \sim x$ with $y \in \Gamma_{k_1}$.
Thus we may assume that
\[y = \langle e_1, \dots, e_{k_1}, e_{k+1}, \dots, e_{2k-k_1} \rangle.\]
If $2k - k_1 + 1 \leq n$, let
\[z = \langle e_1, \dots, e_{k_1}, e_{k+2}, \dots, e_{2k-k_1+1}\rangle\]
so that dim$(x \cap z) = k_1$ and dim$(y \cap z) = k-1 > k_1$.
Assume to the contrary that $k_1 > 0$ so in particular $z$ is defined.
It then follows from (F1) that $x \sim z$ and from Claim 2 and the maximality of $k_1$ that $y \nsim z$.
Since $1 \le k_1 \le k-2$, we have $k + 2 \leq 2k - k_1$ and hence the permutation matrix corresponding to $(1,k+2)(k,k-1)$ fixes $y$ and $z$ but not $x$.
But once again Lemma~\ref{diamondOriginal} implies that $y \sim z$, a contradiction. Therefore $k_1=0$ as claimed.

\smallskip

An immediate corollary of Claim 3 and (F1) is that $G$ acts flag-transitively on $\mathcal{S}$.

\smallskip

{\sc Claim 4}: $n=2k$ or $2k+1$.
For a contradiction, suppose that $2k+1<n$ and recall $k \le n/2$.
Let $y = \langle e_{k+1},\dots,e_{2k} \rangle$ and $z = \langle e_{k+2},\dots,e_{2k+1}\rangle$.
Observe that $x \sim y$, $x \sim z$ by $(F1)$; furthermore dim$(y \cap z) = k-2 >0$, so $y \nsim z$ by the maximality of $k_1$.
Since $k \ge 4$ by Claim 1, the permutation matrix corresponding to $(1,2k+2)(2,3)$ fixes $y$ and $z$ but not $x$, contradicting Lemma~\ref{diamondOriginal}.

\smallskip

{\sc Claim 5}: $n = 2k$.
Assume $n=2k+1$.
Let $y$ be as in Claim 4 and let
$z = \langle e_{k+1}, \dots, e_{2k-1}, e_1 + e_{2k+1}\rangle$.
Then $x \sim y$ and $x \sim z$ by Claim 3 and since $\dim(y \cap z) =k-1>0$, we see that $y \nsim z$ by Claim 3.
Once again we apply Lemma~\ref{diamondOriginal} by noting that since $k \ge 4$, the permutation matrix for $(1,2k+1)(k+1,k+2)$ leaves $y$ and $z$ fixed but not $x$, contradicting Lemma~\ref{diamondOriginal}. Hence $n=2k$ and
  Claim~5 is true.

\smallskip

To complete the proof let
\begin{equation}\label{E:xyz}
x = \langle e_1, \dots, e_k\rangle,\quad
y = \langle e_{k+1}, \dots, e_{2k}\rangle,\quad
z = \langle e_1 + e_{k+1}, \dots, e_i + e_{i+k} ,\dots, e_k+e_{2k} \rangle.
\end{equation}
Then $\dim(x \cap y) = \dim(x \cap z) = \dim(y \cap z) = 0$, and so $x$, $y$ and $z$ are pairwise collinear by Claim 3.
Then, since $\mathcal{S}$ does not contain any triangles, $x$, $y$ and $z$ lie on a line of $\s$, say $\ell$.
Consider the stabiliser $G_\ell$.
Note that $\ell$ is the unique line containing any pair of the elements $x$, $y$ or $z$ and so in particular, $G_\ell \ge \langle G_{xy}, G_{xz}, G_{yz} \rangle$.
Writing vectors in $V$ as $n$-dimensional row vectors over $\mathbb{F}_q$ relative to the basis $e_1,\dots,e_n$, and writing matrices relative to this basis, we see that $x$ consists of all vectors of the form $(X, 0)$, where $X, 0$ denote $k$-dimensional row vectors, and   
the stabiliser $G_x$ consists of all matrices $M = \bigl(\begin{smallmatrix}A & B\\C & D\end{smallmatrix}\bigr) \in G$ for which $(I \mid 0) M$ has the form $(X\mid 0)$ where $X\in\GL_k(\mathbb{F}_q)$.
Our aim is to show that $\langle G_{xy}, G_{xz}, G_{yz} \rangle$ contains $\SL_n(q)$.
Let $H = \SL_n(q)$ and let $H_x = H \cap G_x$ and define $H_y$, $H_z$, $H_{xy}$, $H_{xz}$, $H_{yz}$ and $H_\ell$ analogously. Let $M_k(q)$ denote the ring of all $k\times k$ matrices over $\mathbb{F}_q$.
Then
\[H_x = \langle \bigl(\begin{smallmatrix}A & 0\\C & D\end{smallmatrix}\bigr) \in \SL_n(q) \mid A, D \in \GL_k(q), C \in M_k(q)\rangle.\]
Similarly,
\[H_y = \langle \bigl(\begin{smallmatrix}A & B\\0 & D\end{smallmatrix}\bigr) \in \SL_n(q) \mid A, D \in \GL_k(q), B \in M_k(q)\rangle\]
and
\[H_z = \langle \bigl(\begin{smallmatrix}A & B\\C & D\end{smallmatrix}\bigr) \in \SL_n(q) \mid A+C = B+D \rangle.\]
From this we see that
\[H_{xy} = \langle \bigl(\begin{smallmatrix}A & 0\\0 & D\end{smallmatrix}\bigr) \in \GL_n(q) \mid A, D \in \GL_k(q), \det(AD)=1\rangle.\]
Similarly,
\[H_{xz} = \langle \bigl(\begin{smallmatrix}A & 0\\D-A & D\end{smallmatrix}\bigr) \in \GL_n(q) \mid A, D \in \GL_k(q), \det(AD)=1\rangle\]
and
\[H_{yz} = \langle \bigl(\begin{smallmatrix}A & A-D\\0 & D\end{smallmatrix}\bigr) \in \GL_n(q) \mid A, D \in \GL_k(q), \det(AD)=1\rangle.\]
Our aim is now to show that $H_\ell := \langle H_{xy}, H_{xz}, H_{yz}\rangle$ is equal to $H$.
We interrupt our proof of Theorem~\ref{main} to prove this in the following lemma.

\begin{lemma}\label{L:SL}
  The group generated by all matrices of the form
$\bigl(\begin{smallmatrix}A & 0\\0 & D\end{smallmatrix}\bigr)$,
$\bigl(\begin{smallmatrix}A & A-D\\0 & D\end{smallmatrix}\bigr)$ and
$\bigl(\begin{smallmatrix}A & 0\\D-A & D\end{smallmatrix}\bigr)$
where $A,D \in \GL_k(q)$ and $\det(AD)=1$ equals $H \cong \SL_{2k}(q)$.
\end{lemma}

\begin{proof}
Let $L = \langle
\bigl(\begin{smallmatrix}A & 0\\0 & D\end{smallmatrix}\bigr),
\bigl(\begin{smallmatrix}A & A-D\\0 & D\end{smallmatrix}\bigr),
\bigl(\begin{smallmatrix}A & 0\\D-A & D\end{smallmatrix}\bigr) \mid A, D \in \GL_k(q), \det(AD) = 1 \rangle$
and let $x, y, z\in\binom{V}{k}$ be as in~\eqref{E:xyz}.
Then, $L$ contains the matrix
$\bigl(\begin{smallmatrix}A & 0\\D-A & D\end{smallmatrix}\bigr)
\bigl(\begin{smallmatrix}A^{-1} & 0\\0 & D^{-1}\end{smallmatrix}\bigr)
 =
\bigl(\begin{smallmatrix}I & 0\\ DA^{-1}-I & I\end{smallmatrix}\bigr)$.
In particular, choosing $A = I$ and $D = I + E_{1,2}$ we have $DA^{-1} = I + E_{1,2}$ so that $L$ contains the matrix
$M=\bigl(\begin{smallmatrix}I & 0\\ E_{1,2} & I\end{smallmatrix}\bigr)$.
An element $h = h_{A,D} = \bigl(\begin{smallmatrix}A & 0\\0 & D\end{smallmatrix}\bigr)$ conjugates
$M=\bigl(\begin{smallmatrix}I & 0\\ E_{1,2} & I\end{smallmatrix}\bigr)$
to
$\bigl(\begin{smallmatrix}I & 0\\ D^{-1} E_{1,2} A & I\end{smallmatrix}\bigr)$.
Since $L$ contains
$\bigl(\begin{smallmatrix}A & 0\\0 & A\end{smallmatrix}\bigr)$
for each permutation matrix $A$, it follows that $L$ contains
$\bigl(\begin{smallmatrix}I & 0\\ E_{i,j} & I\end{smallmatrix}\bigr)$
for each $i,j$.
Hence $L$ contains $H_x$.
Similarly, $L$ contains $H_y$.
Since $H_x$ is maximal in $H$ and $H_x \neq H_{y}$, we conclude that $L = H$.
\end{proof}

Resuming our proof: Lemma~\ref{L:SL} implies that $G_\ell$ contains $\SL_n(q)$ and since $G$ is primitive on points it follows that $\SL_n(q)$ and hence also~$G_\ell$, is transitive on points.
This implies that $\ell$ is incident with all points, which is a contradiction.
This completes the proof of Theorem \ref{main}.
\end{proof}

\section*{Acknowledgements}
The authors gratefully acknowledge support from the Australian Research
Council Discovery Project Grants DP140100416 and DP190100450.
This research was conducted at the University of Western Australia
whose campus we acknowledge is located on the traditional lands of
the Whadjuk people of the Noongar Nation. We pay our respects to
Noongar Elders past, present and emerging.

\bibliographystyle{amsrefs}
\bibliography{bib}

\end{document}